\documentclass[12pt]{amsart}

\usepackage{amsmath,amssymb,amsthm}
\usepackage{mathrsfs,here}
\usepackage[all]{xy}
\usepackage[dvipdfmx,colorlinks=true]{hyperref}

\numberwithin{equation}{section}

\SelectTips{eu}{12}

\newtheorem{theorem}{Theorem}[section]
\newtheorem{proposition}[theorem]{Proposition}
\newtheorem{lemma}[theorem]{Lemma}
\newtheorem{corollary}[theorem]{Corollary}

\theoremstyle{definition}
\newtheorem{definition}[theorem]{Definition}

\theoremstyle{remark}
\newtheorem{remark}[theorem]{Remark}

\newcommand{\Z}{\mathbb{Z}}
\newcommand{\G}{\mathscr{G}}
\newcommand{\SU}{\operatorname{SU}}
\newcommand{\Sp}{\operatorname{Sp}}
\newcommand{\Spin}{\operatorname{Spin}}
\newcommand{\adjoint }{\operatorname{ad}}

\title{Infiniteness of $A_\infty$-types of gauge groups}

\author{Daisuke Kishimoto}
\address{Department of Mathematics, Kyoto University, Kyoto, 606-8502, Japan}
\email{kishi@math.kyoto-u.ac.jp}
\author{Mitsunobu Tsutaya}
\address{Department of Mathematics, Kyoto University, Kyoto, 606-8502, Japan}
\email{tsutaya@math.kyoto-u.ac.jp}
\subjclass[2010]{Primary 57S05; Secondary 55P48}
\keywords{gauge group, $A_\infty$-space, $A_\infty$-type}

\begin{document}

\maketitle

\begin{abstract}
Let $G$ be a compact connected Lie group and let $P$ be a principal $G$-bundle over $K$. The gauge group of $P$ is the topological group of automorphisms of $P$. For fixed $G$ and $K$, consider all principal $G$-bundles $P$ over $K$. It is proved  in \cite{CS,Ts} that the number of $A_n$-types of the gauge groups of $P$ is finite if $n<\infty$ and $K$ is a finite complex. We show that the number of $A_\infty$-types of the gauge groups of $P$ is infinite if $K$ is a sphere and there are infinitely many $P$.
\end{abstract}

% baseline
\baselineskip 16pt

%%%%%%%%%%%%%%%%%%
%  Section 1
%%%%%%%%%%%%%%%%%%

\section{Introduction}

Let $G$ be a compact connected Lie group and let $P$ be a principal $G$-bundle over a base $K$.
The \textit{gauge group} of $P$, denoted by $\G(P)$, is the topological group of all automorphisms of $P$, i.e. $G$-equivariant self-maps of $P$ covering the identity map of $K$, where the multiplication is given by the composite of maps.
In \cite{CS}, Crabb and Sutherland pose the following problem.
For fixed $K$ and $G$, consider all principal $G$-bundles $P$ over $K$: is the number of homotopy types of $\G(P)$, or of $B\G(P)$, finite?
Of course the problem makes sense when there are infinitely many principal $G$-bundles over $K$, or equivalently, the homotopy set $[K,BG]$ is of infinite order.
The first example of this problem is formerly considered by Kono \cite{K}: when $G=\mathrm{SU}(2)$ and $K=S^4$, there are exactly 6 homotopy types of $\G(P)$ while $\pi_4(B\mathrm{SU}(2))$ is of infinite order.
This result is actually a consequence of the complete classification of homotopy types of $\G(P)$ for $G=\mathrm{SU}(2)$ and $K=S^4$.
There are several such examples \cite{HK1,HK2,KKKT,Th}.
In a general setting, Crabb and Sutherland \cite{CS} prove that when $K$ is a finite complex, the number of homotopy types of $\G(P)$ is finite.
They actually prove a stronger result that when $K$ is a finite complex, the number of H-types (i.e. homotopy types as H-spaces) of $\G(P)$ is finite.
Recall that there are intermediate classes of spaces between H-spaces and loop spaces, called {\it$A_n$-spaces} \cite{St}, where $H$-spaces are $A_2$-spaces and loop spaces are $A_\infty$-spaces.
So it is natural to count the number of {\it$A_n$-types} (i.e. homotopy types as $A_n$-spaces) of gauge groups, and the second named author \cite{Ts} generalizes the above result of Crabb and Sutherland to $A_n$-types: when $K$ is a finite complex and $n<\infty$, the number of $A_n$-types of $\G(P)$ is finite.
This generalization makes (in)finiteness of the number of homotopy types of $B\G(P)$, or equivalently, $A_\infty$-types of $\G(P)$, more meaningful.
As for $A_\infty$-types of $\G(P)$, there is only one example due to Masbaum \cite{Ma}: for $G=\mathrm{SU}(2)$ and $K=S^4$ (the same situation as Kono's example) the number of $A_\infty$-types of $\G(P)$ is infinite.
This is a consequence of the cohomology calculation of $B\G(P)$ for the above special $K$ and $G$.
In this paper, we prove infiniteness of the number of $A_\infty$-types of $\G(P)$ in a more general setting by investigating $A_n$-types of gauge groups.

\begin{theorem}
\label{main}
Let $G$ be a compact connected simple Lie group. As $P$ ranges over all principal $G$-bundles over $S^d$, if there are infinitely many isomorphism types for $P$, then there are also infinitely many $A_\infty$-types of gauge groups $\G(P)$.
\end{theorem}

As in \cite{G} and \cite{AB}, it is well-known that the connected component of the mapping space $\mathrm{map}(K,BG)$ containing a map $\alpha\colon K\to BG$ has the weak homotopy type of the classifying space of the gauge group $\mathscr{G}(\alpha^\ast EG)$.
Hence the above theorem also implies that infinitely many different weak homotopy types appear among the connected components of $\mathrm{map}(S^d,BG)$.

This work began during the authors' stay at University of Southampton.
They would like to thank University of Southampton for hospitality, and would also like to thank Stephen Theriault for invitation and fruitful conversation.
The first author was partially supported by JSPS KAKENHI Grant Number 25400087.

%%%%%%%%%%%%%%%%%%
%  Section 2
%%%%%%%%%%%%%%%%%%

\section{Preliminaries}

This section collects properties of $A_n$-maps and unstable algebras over the Steenrod operations which we will employ.

%%%%%%%%%%%%%%%%%%
%  subsection 2.1
%%%%%%%%%%%%%%%%%%

\subsection{$A_n$-maps}

In this paper, $A_n$-maps mean $A_n$-maps between topological monoids while there is a generalized definition of $A_n$-maps between $A_n$-spaces.
This section collects basics of (fiberwise) $A_n$-maps between (fiberwise) topological monoids.
We first recall the definition of $A_n$-maps between topological monoids due to Stasheff, where our references for $A_n$-maps between topological monoids are \cite{St} and \cite{Fu}.

\begin{definition}
A map $f\colon X\to Y$ between topological monoids $X,Y$ is called an \textit{$A_n$-map} if there is a collection of maps $\{h_i\colon I^{i-1}\times X^i\to Y\}_{i=1}^n$ satisfying $h_1=f$ and for $i>1$
\begin{multline*}
h_i(t_1,\ldots,t_{i-1};x_1,\ldots,x_i)\\
=\begin{cases}h_{i-1}(t_1,\ldots,\widehat{t_j},\ldots,t_{i-1};x_1,\ldots,x_jx_{j+1},\ldots,x_i)&t_j=0\\
h_j(t_1,\ldots,t_{j-1};x_1,\ldots,x_j)h_{i-j}(t_{j+1},\ldots,t_{i-1};x_{j+1},\ldots,x_i)&t_j=1.\end{cases}
\end{multline*}
\end{definition}

A homotopy equivalence between topological monoids which is an $A_n$-map is called an \textit{$A_n$-equivalence}.
$A_n$-maps between topological monoids have the following properties, implying that $A_n$-equivalences yield an equivalence relation.
The equivalence classes are called \textit{$A_n$-types}.

\begin{proposition}
\label{A_n-property}
\begin{enumerate}
\item The composite of $A_n$-maps is an $A_n$-map.
\item The homotopy inverse of a homotopy equivalence which is an $A_n$-map is an $A_n$-map.
\end{enumerate}
\end{proposition}

In \cite{KK}, a straightforward generalization of $A_n$-maps to fiberwise topological monoids is introduced, and it is shown that they have analogous properties in Proposition \ref{A_n-property}.
Then a fiberwise homotopy equivalence between fiberwise topological monoids which is a fiberwise $A_n$-map is called a \textit{fiberwise $A_n$-equivalence}, and fiberwise $A_n$-equivalences yield an equivalence relation.

We now return to the usual $A_n$-maps.
It is complicated to verify that a given map is an $A_n$-map by checking definition.
There is a useful equivalence condition for a map being an $A_n$-map.
For a topological monoid $X$, let $P^nX$ be the $n$-th projective space constructed by $n$-times iterated use of the Dold-Lashof construction.
We denote the classifying space of $X$ by $BX$, i.e. $BX=P^\infty X$.
If a map $f\colon X\to Y$ between topological monoids $X,Y$ is an $A_n$-map, it induces a map $P^nf\colon P^nX\to P^nY$ satisfying the homotopy commutative diagram
\begin{equation}
\label{P^nf}
\xymatrix{\Sigma X\ar[r]^{\Sigma f}\ar[d]&\Sigma Y\ar[d]\\
P^nX\ar[r]^{P^nf}&P^nY}
\end{equation}
where the vertical arrows are the canonical maps.
The following proposition shows that the existence of a map like $P^nf$ is a necessary and sufficient condition for $f$ being an $A_n$-map.

\begin{proposition}
\label{P^nX}
Let $X$ and $Y$ be grouplike topological monoids.
Then, a map $f\colon X\to Y$ is an $A_n$-map if and only if there is a map $\bar{f}\colon P^nX\to BY$ satisfying the homotopy commutative diagram
$$\xymatrix{\Sigma X\ar[r]^{\Sigma f}\ar[d]&\Sigma Y\ar[d]\ar[d]\\
P^nX\ar[r]^{\bar{f}}&BY}$$
where the vertical arrows are the canonical maps.
\end{proposition}

%%%%%%%%%%%%%%%%%%
%  subsection 2.2
%%%%%%%%%%%%%%%%%%

\subsection{Unstable algebras}

Throughout this subsection, the ground field of algebras, including the Steenrod algebra, are $\Z/p$, and maps between algebras with actions of the Steenrod algebra, not necessarily unstable, preserve these actions.  We first recall results on embeddings of unstable algebras into $H^*(BT^\ell;\Z/p)$ due to Adams and Wilkerson which we will employ, where we refer to \cite{AW} for more precise statements. Algebraic and integral extensions of graded algebras are defined in the same manner of usual rings by using homogeneous polynomials. We consider the following condition of graded algebras:
\begin{equation}
\label{condition}
K^*\text{ is a connected finitely generated graded domain and }K^{\rm odd}=0.
\end{equation}
Unstable algebras satisfying \eqref{condition} in mind are the cohomology $H^*(BG;\Z/p)$ for a connected compact Lie group $G$ without $p$-torsion in integral homology and its unstable subalgebras. 

\begin{theorem}
[Adams and Wilkerson {\cite[Theorem 1.1 and 1.6]{AW}}]
\label{AW1}
If an unstable algebra $K^*$ satisfies the condition \eqref{condition}, there is an algebraic extension 
$$K^*\to H^*(BT^\ell;\Z/p)$$ 
for some $\ell$, which is unique up to automorphisms of $H^*(BT^\ell;\Z/p)$.
\end{theorem}

\begin{theorem}
[Adams and Wilkerson {\cite[Proposition 1.10]{AW}}]
\label{AW2}
Let $K^*\to H^*(BT^k;\Z/p)$ and $L^*\to H^*(BT^\ell;\Z/p)$ be algebraic extensions of unstable algebras $K^*,L^*$ satisfying \eqref{condition}. Then for any map $K^*\to L^*$, there is a dotted filler in the commutative diagram
$$\xymatrix{K^*\ar[r]\ar[d]&H^*(BT^k;\Z/p)\ar@{-->}[d]\\
L^*\ar[r]&H^*(BT^\ell;\Z/p).}$$ 
\end{theorem}

We will need to convert algebraic extensions of unstable algebras to integral extensions in order to compare the Krull dimensions. This is possible by the modification of the results of Wilkerson \cite{W} in \cite{N} which we recall here. We set notation. Let $K^*$ be a graded algebra satisfying \eqref{condition} on which the mod $p$ Steenrod algebra acts (not necessarily unstably). As in \cite{N} (cf. \cite{W}), the unstable part of $K^*$ is defined by
$$\mathscr{U}(K^*):=\{x\in K^*\,\vert\,\mathscr{P}^r\mathscr{P}^Ix=0\text{ for }2r>\deg\mathscr{P}^I+|x|\text{ for any multi-index }I\},$$
where for a multiindex $I=(i_1,\ldots,i_s)$, we put $\mathscr{P}^I:=\mathscr{P}^{i_1}\cdots\mathscr{P}^{i_s}$. 

\begin{proposition}
[{Neusel \cite[Proposition 2.3]{N}}]
\label{unstable-part1}
Let $K^*,L^*$ be graded algebras satisfying \eqref{condition} on which the Steenrod algebra acts, not necessarily unstably. If $K^*\to L^*$ is an integral extension, then $\mathscr{U}(K^*)\to\mathscr{U}(L^*)$ is an integral extension.
\end{proposition}

As in \cite{W}, if a graded algebra has an action of the Steenrod algebra, then its ring of fractions inherits the action.
In the next proposition, the field of fractions $Q(K^*)$ of an unstable algebra $K^*$ is considered as the ``homogeneous portion'' in \cite[Proposition 2.3]{W} or the ``field of fractions'' in \cite{N}.

\begin{proposition}
[{Neusel \cite[Theorem 2.4]{N}, Wilkerson \cite[Proposition 3.3]{W}}]
\label{unstable-part2}
For a graded algebra $K^*$, let $\overline{K^*}$ denote the integral closure of $K^*$ in its field of fractions $Q(K^*)$. 
\begin{enumerate}
\item If $K^*$ is an unstable algebra satisfying \eqref{condition}, then
$$\mathscr{U}(Q(K^*))=\overline{K^*}.$$
\item If $K^*$ is an unstable algebra which is a UFD satisfying \eqref{condition}, then
$$\mathscr{U}(Q(K^*))=K^*.$$
\end{enumerate}
\end{proposition}

\begin{remark}
The definitions of the unstable parts in \cite{W} is slightly different from ours \cite{N}, but we can prove Proposition \ref{unstable-part2} in the same way as Proposition 3.3 of \cite{W}.
\end{remark}

\begin{corollary}
\label{dimension}
Let $K^*$ be an unstable algebra satisfying \eqref{condition}. Then any algebraic extension $K^*\to H^*(BT^\ell;\Z/p)$ is an integral extension.
\end{corollary}

\begin{proof}
Put $L^*:=H^*(BT^\ell;\Z/p)$. Then $Q(K^*)\to Q(L^*)$ is an algebraic extension of fields, so it is an integral extension. Then it follows from Proposition \ref{unstable-part1} that $\mathscr{U}(Q(K^*))\to\mathscr{U}(Q(L^*))$ is an integral extension. Now by Proposition \ref{unstable-part2},  we have $\mathscr{U}(Q(K^*))=\overline{K^*}$ and $\mathscr{U}(Q(L^*))=L^*$, so we get a sequence of integral extensions
$$K^*\to\overline{K^*}\to L^*.$$
Then the composite is an integral extension which can be identified with the original algebraic extension $K^*\to L^*$ by the uniqueness of Theorem \ref{AW1}, completing the proof.
\end{proof}

We next recall Aguad\'e's calculation of the $T$-functor. Let $G$ be a compact connected Lie group such that $H_*(G;\Z)$ is $p$-torsion free. Then the mod $p$ cohomology of $BG$ is isomorphic to the invariant ring of the action of the Weyl group of $G$ on $H^*(BT;\Z/p)$, where $T$ is a maximal torus of $G$. Let $V$ be the canonical elementary abelian $p$-subgroup in $T$, and let $j\colon BV\to BG$ denote the canonical map. We denote the $T$-functor associated with $V$ by $T_V$.

\begin{proposition}
[Aguad\'e {\cite[Proposition 4]{A}}, Lannes {\cite[Proposition 3.4.6]{L}}]
\label{A}
There is an isomorphism
$$T_V^{j^*}H^*(BG;\Z/p)\cong H^*(BT;\Z/p)$$
where the left hand side denotes the component of $j^*$ in $T_VH^*(BG;\Z/p)$.
\end{proposition}

%%%%%%%%%%%%%%%%%%
%  Section 3
%%%%%%%%%%%%%%%%%%

\section{Proof of Theorem \ref{main}}

The outline of the proof of Theorem \ref{main} is as follows. Let $G$ be a compact connected Lie group such that $\pi_{d-1}(G)$ is of infinite order, and let $P$ be a principal $G$-bundle over $S^d$ which is classified by $\alpha\in\pi_{d-1}(G)$. Note that the infiniteness of $\pi_{d-1}(G)$ is equivalent to that there are infinitely many $P$. We first prove that given a positive $n$, if $\alpha$ is divisible by $p^N$ for $N$ large, then $\G(P)_{(p)}$ and $\G(S^d\times G)_{(p)}$ have the same $A_n$-type. We next prove that if $\G(P)_{(p)}$ and $\G(S^d\times G)_{(p)}$ have the same $A_n$-type for $n,p$ large and $G$ simple, then $\alpha$ is divisible by $p$. Then since there are infinitely many primes, we conclude that there are infinitely many $A_\infty$-types of $\G(P)$ when $P$ ranges over all principal $G$-bundles over $S^d$. In the both parts, the obstruction theoretic description of the $A_n$-triviality of the adjoint bundle of $P$ as in \cite{KK} is fundamental, and in the second part, the results of unstable algebras in Section 2 underlie technical arguments.

To recall the above mentioned result of Kono and the first author \cite{KK}, we briefly recall fiberwise objects appearing in it. A space $X$ equipped with a map $p\colon X\to B$ is called a {\it fiberwise space} over a space $B$. The given map $p$ is called the {\it projection}. A map $X\to Y$ between fiberwise spaces over $B$ is called a {\it fiberwise map} if it commutes with projections. For a fiberwise map $m\colon X\times_BX\to X$ from the fiber product of $X$ and itself and a section $s\colon B\to X$, the triple $X=(X,m,s)$ is called a {\it fiberwise topological monoid} if $m\circ(1\times m)=m\circ(m\times 1)$ as maps $X\times_BX\times_BX\to X$ and $m\circ(1,s\circ p)=m\circ(s\circ p,1)=1$ as maps $X\to X$. We say a fiberwise topological monoid $X$ is a {\it fiberwise topological group} if there is a fiberwise map $\nu\colon X\to X$ such that $m\circ(1,\nu)=m\circ(\nu,1)=s\circ p$. Like $A_n$-maps and $A_n$-equivalences between topological monoids, {\it fiberwise $A_n$-maps} and {\it fiberwise $A_n$-equivalences} between fiberwise topological monoids are defined.

Let $G$ be a compact connectd Lie group, and let $P$ be a principal $G$-bundle over a base $K$.  Define $\adjoint P$ by
$$\adjoint P=(P\times G)/\sim$$
where $(x,g)\sim(xh,h^{-1}gh)$ for $x\in P$ and $g,h\in G$. Then the projection $\adjoint P\to K$ is a fiber bundle over $K$ with fiber $G$ which we call the {\it adjoint bundle} of $P$. As in \cite{KK}, $\adjoint P$ is a fiberwise topological group over $K$, and then in particular, the space of all sections $\Gamma(\adjoint P)$ becomes a topological group by the pointwise multiplication. Let $\adjoint P_{(p)}$ denote the fiberwise $p$-localization of $\adjoint P$. By the standard Moore path technique, we may assume that $\adjoint P_{(p)}$ is a fiberwise topological monoid as well as $G_{(p)}$. It is shown in \cite{AB} that there is a natural isomorphism of topological groups
$$\G(P)\cong\Gamma(\adjoint P)$$
from which we get an $A_\infty$-equivalence 
$$\G(P)_{(p)}\simeq\Gamma(\adjoint P_{(p)}).$$
We connect the triviality of $\adjoint P$ with the gauge group of $P$. The restriction to the fiber at the basepoint yields a homomorphism of topological groups
$$\pi\colon\G(P)\to G$$
which is identified with the map $\Gamma(\adjoint P)\to G$ substituting the basepoint of $K$. If $\pi$ has a right homotopy inverse $\sigma$, the map
$$K\times G\to\adjoint P,\quad(x,g)\mapsto\sigma(g)(x)$$
is a fiberwise homotopy equivalence, where we regard $\sigma(g)$ as a section of $\adjoint P$ by the above isomorphism. In \cite{KK} this is generalized in the context of fiberwise $A_n$-types. For a given $A_\infty$-map, we call its right homotopy inverse which is an $A_n$-map by an $A_n$-section.

\begin{theorem}
[\cite{KK}]
\label{A_k-section}
The adjoint bundle $\adjoint P$ is fiberwise $A_n$-equivalent to the trivial bundle $K\times G$ if and only if $\pi$ has an $A_n$-section.
\end{theorem}

Let $P_\alpha$ denote the principal $G$-bundle classified by a map $\alpha\colon K\to BG$. As in \cite{G} and \cite{AB}, there is a natural homotopy equivalence
$$B\G(P_\alpha)\simeq\mathrm{map}(K,BG;\alpha)$$
where $\mathrm{map}(X,Y;f)$ denotes the space of maps from $X$ to $Y$ which are (freely) homotopic to $f$. Evaluating at the basepoint of $K$, we get the fibration
$$\omega\colon\mathrm{map}(K,BG;\alpha)\to BG$$
with fiber $\mathrm{map}_0(K,BG;\alpha)$, where $\mathrm{map}_0(X,Y;f)$ is the space of basepoint preserving maps from $X$ to $Y$ which are freely homotopic to $f$.
By construction, the map $\pi\colon\G(P_\alpha)\to G$ is identified with $\Omega\omega$ through the above homotopy equivalence. Using this identification, we show an obstruction theoretic description of existence of an $A_n$-section of $\pi$.
Let $j_n\colon\Sigma G\to P^nG$ and $i_1\colon\Sigma G\to BG$ be the canonical maps.

\begin{lemma}
[cf. \cite{KK}]
\label{obstruction}
The map $\pi\colon\G(P_\alpha)\to G$ has an $A_n$-section if and only if there is a map $K\times P^nG\to BG$ satisfying the homotopy commutative diagram
$$
\xymatrix{
K\vee \Sigma G\ar[r]^(.6){\alpha\vee i_1}\ar[d]^{\rm incl}&BG\ar@{=}[d]\\
K\times P^nG\ar[r]&BG.
}
$$
\end{lemma}

\begin{proof}
As above, we identify $\pi$ with $\Omega\omega$.
By Lemma \ref{P^nX}, there is an $A_n$-section of $\Omega\omega$ if and only if there are maps $s\colon G\to\Omega\mathrm{map}(K,BG;\alpha)$ and $\bar{s}\colon P^nG\to\mathrm{map}(K,BG;\alpha)$ satisfying $\Omega\omega\circ s\simeq 1_G$ and the homotopy commutative diagram
$$\xymatrix{\Sigma G\ar[r]^(.3){\Sigma s}\ar[d]^{j_n}&\Sigma\Omega\mathrm{map}(K,BG;\alpha)\ar[d]\\
P^nG\ar[r]^(.35){\bar{s}}&\mathrm{map}(K,BG;\alpha).}$$
Thus by taking adjoint, we obtain the desired result.
\end{proof}

\begin{remark}
We here remark that Theorem \ref{A_k-section} and Lemma \ref{obstruction} hold if we localize at $p$.
\end{remark}

We first show an implication of the divisibility of a map $\alpha$ on the triviality of the adjoint bundle $\adjoint P_\alpha$. From now on we set $K=S^d$. Then we have $\alpha\in\pi_d(BG)\cong\pi_{d-1}(G)$. 

\begin{lemma}
\label{trivial}
For any $f\colon X\to BG$, the connecting map $\partial\colon\pi_*(BG)\to\pi_{*-1}(\mathrm{map}_0(X,BG;f))$ of the evaluation fibration is trivial after rationalization.
\end{lemma}

\begin{proof}
Since $G$ is a Lie group, $BG_{(0)}$ is a product of Eilenberg-MacLane spaces, so it is in particular an H-space. Then $\mathrm{map}(X,BG;f)_{(0)}$ is fiberwise homotopy equivalent to $BG_{(0)}\times\mathrm{map}_0(X,BG;f)_{(0)}$, implying that the map $\omega_*\colon\pi_*(\mathrm{map}(X,BG;f))_{(0)}\to\pi_*(BG)_{(0)}$ is surjective. Thus the proof is done.
\end{proof}

\begin{proposition}
\label{divisibility-1}
For given $n$, there is an integer $N$ such that if $\alpha\in\pi_d(BG)$ is divisible by $p^N$, then the fiberwise $p$-localized adjoint bundle $(\adjoint P_\alpha)_{(p)}$ is fiberwise $A_n$-equivalent to the trivial bundle $S^d\times G_{(p)}$.
\end{proposition}

\begin{proof}
It is sufficient to show that if $\alpha$ is divisible by $p^N$ for some $N$,  then there is a map $\mu\colon S^d_{(p)}\times P^nG_{(p)}\to BG_{(p)}$ which restricts to $\alpha\vee(i_n)_{(p)}\colon S^d_{(p)}\vee P^nG_{(p)}\to BG_{(p)}$ up to homotopy, where $i_n\colon P^nG\to BG$ is the canonical inclusion. Indeed for $i_n\circ j_n\simeq i_1$, existence of $\mu$ implies that $(\adjoint P_\alpha)_{(p)}$ is fiberwise $A_n$-equivalent to $K\times G_{(p)}$ by Lemma \ref{obstruction}. By adjointness the map $\mu$ exists if and only if $\alpha_{(p)}$ is in the image of the induced map $\omega_*\colon\pi_d(\mathrm{map}(P^nG_{(p)},BG_{(p)};i_n))\to\pi_d(BG_{(p)})$, which is equivalent to that $\partial(\alpha_{(p)})=0$ for the connecting map $\partial\colon\pi_d(BG)_{(p)}\to\pi_{d-1}(\mathrm{map}_0(P^nG,BG;i_n))_{(p)}$. By Lemma \ref{trivial} this connecting map $\partial$ is trivial if we rationalize, so its image is of finite order. Then if $\alpha$ is divisible by $p^N$ for $N$ large, we have $\partial(\alpha_{(p)})=0$, completing the proof.
\end{proof}

We next show an implication of the triviality of $\adjoint P_\alpha$ on the divisibility of the classifying map $\alpha$. By the Hopf theorem, we have a rational homotopy equivalence
$$G\simeq_{(0)}S^{2n_1-1}\times\cdots\times S^{2n_\ell-1}$$ 
for $n_1\le\cdots\le n_\ell$, where $n_1,\ldots,n_\ell$ is called the type of $G$. For $p>n_\ell$, there is also a $p$-local homotopy equivalence
\begin{equation}
\label{p-regular}
G\simeq_{(p)}S^{2n_1-1}\times\cdots\times S^{2n_\ell-1}.
\end{equation}
Hereafter we always assume $p>n_\ell$. Notice that the mod $p$ cohomology of $BG$ is given by
$$H^*(BG;\Z/p)\cong\Z/p[y_{2n_1},\ldots,y_{2n_\ell}],\quad|y_{2n_i}|=2n_i.$$
Here we regard that $y_{2n_i}$ corresponds to the sphere $S^{2n_i-1}$ in \eqref{p-regular}, so even if $n_i=n_j$ for some $i,j$, we distinguish $y_{2n_i}$ and $y_{2n_j}$. The following replacement of maps  between unstable algebras allows us to apply the results in Section 2.

\begin{lemma}
\label{lift}
If $n\ge n_\ell+p-1$ and $K^*$ is an unstable algebra, then for any map $\phi\colon H^*(BG;\Z/p)^*\to K^*\otimes H^*(P^nG;\Z/p)$ of unstable algebras, there is a dotted filler in the commutative diagram of unstable algebras
$$\xymatrix{H^*(BG;\Z/p)\ar@{-->}[r]\ar@{=}[d]&K^*\otimes H^*(BG;\Z/p)\ar[d]^{1\otimes i_n^*}\\
H^*(BG;\Z/p)\ar[r]^(.43)\phi&K^*\otimes H^*(P^nG;\Z/p)}$$
\end{lemma}

\begin{proof}
The homotopy fiber of $i_n$ is the join of $(n+1)$-copies of $G$ whose mod $p$ cohomology is trivial in dimension $\le 2n$ since $G$ is connected. Then we see that $i_n$ is an isomorphism in mod $p$ cohomology of dimension $\le 2n$, implying that the dotted filler, say $\varphi$, exists as a maps of graded algebras. The map $\varphi$ actually respects the Steenrod operations. Indeed, for $n\ge n_\ell+p-1$, 
$$\varphi(\theta y_{2n_i})=(1\otimes i_n^*)^{-1}\phi(\theta y_{2n_i})=\theta(1\otimes i_n^*)^{-1}\phi(y_{2n_i})=\theta\varphi(y_{2n_i})\quad\text{for}\quad\theta=\beta,\mathscr{P}^1$$
and since we are assuming $p>n_\ell$, 
$$\varphi(\mathscr{P}^{p^k}y_{n_i})=0=\mathscr{P}^{p^k}\varphi(y_{n_i})\quad\text{for}\quad k\ge 1.$$
Then since the Steenrod algebra is generated by $\beta,\mathscr{P}^{p^k}$ for $k\ge 0$, the proof is completed.
\end{proof}

\begin{proposition}
\label{divisibility-2}
Suppose $d\ge 4$ and $\alpha$ is of infinite order. If the fiberwise $p$-localized adjoint bundle $(\adjoint P_\alpha)_{(p)}$ is fiberwise $A_n$-equivalent to $S^d\times G_{(p)}$ for $n=n_\ell+p-1$, then $\alpha$ is divisible by $p$.
\end{proposition}

\begin{proof}
We show a contradiction by assuming $\alpha$ is not divisible by $p$. Since $\alpha$ is of infinite order, $d=2n_i$ for some $i$. Then by \eqref{p-regular}, we have $\alpha^*(y_{2n_i})=u$, where $u$ is a generator of $H^d(S^d;\Z/p)$. Since $(\adjoint P_\alpha)_{(p)}$ is fiberwise $A_n$-equivalent to $S^d\times G_{(p)}$, there is a map $\mu\colon S^d_{(p)}\times P^nG_{(p)}\to BG_{(p)}$ which restricts to $\alpha_{(p)}\vee(i_1)_{(p)}\colon S^d_{(p)}\vee\Sigma G_{(p)}\to BG_{(p)}$ up to homotopy by Lemma \ref{obstruction}. Then for $n=n_\ell+p-1$, it follows from Lemma \ref{lift} that there is a map of unstable algebras $\Phi\colon H^*(BG;\Z/p)\to H^*(S^d;\Z/p)\otimes H^*(BG;\Z/p)$ satisfying $(1\otimes i_n^*)\circ\Phi=\mu^*$, so we have that the composite
$$H^*(BG;\Z/p)\xrightarrow{\Phi}H^*(S^d;\Z/p)\otimes H^*(BG;\Z/p)\xrightarrow{\rm proj}H^*(BG;\Z/p),$$
say $\phi$, satisfies $i_n^*\circ\phi=i_n^*$. Hence $\phi$ is an isomorphism between modules of indecomposables, implying that $\phi$ is an isomorphism of unstable algebras. Thus the composite
$$H^*(BG;\Z/p)\xrightarrow{\Phi}H^*(S^d;\Z/p)\otimes H^*(BG;\Z/p)\xrightarrow{1\otimes\phi^{-1}}H^*(S^d;\Z/p)\otimes H^*(BG;\Z/p)$$
denoted by $\Psi$ is a map of unstable algebras which projects to the identity map of $H^*(BG;\Z/p)$ and satisfies $\Psi(y_{2n_i})=u\otimes 1+\text{other terms}$.

Let $T$ be a maximal torus of $G$ and let $V$ be the canonical elementary abelian $p$-subgroup of $T$ of rank $\ell$. Consider the composite of the maps of unstable algebras
$$H^*(BG;\Z/p)\xrightarrow{\Psi}H^*(S^d;\Z/p)\otimes H^*(BG;\Z/p)\xrightarrow{1\otimes j^*}H^*(S^d;\Z/p)\otimes H^*(BV;\Z/p)$$
where $j\colon BV\to BG$ is the canonical map. Then by taking adjoint, there is a map
$$\varphi\colon T_V^{j^*}H^*(BG;\Z/p)\to H^*(S^d;\Z/p)$$
such that the composite of $\varphi$ and the natural map $q\colon H^*(BG;\Z/p)\to T_V^{j^*}H^*(BG;\Z/p)$ sends $y_{2n_i}$ to $u$, where $T_V$ means the Lannes $T$-functor associated with $V$ and $T_V^{j^*}H^*(BG;\Z/p)$ is the component of $j^*$ in $T_VH^*(BG;\Z/p)$ as in Section 2. This is a contradiction by Proposition \ref{A} since $d\ge 4$.
\end{proof}

As in the proof of \cite[Theorem 1.2]{KK}, the $A_n$-triviality of $(\adjoint P_\alpha)_{(p)}$ implies that $\G(P_\alpha)_{(p)}$ and $\G(S^d\times G)_{(p)}$ have the same $A_n$-type. We show the converse under some conditions. For this, we investigate self $A_n$-maps of simple Lie groups.

\begin{lemma}
\label{dim}
If $G$ is simple and $K^*$ is a finitely generated non-trivial unstable subalgebra of $H^*(BG;\Z/p)$, then 
$$\dim_\mathrm{Krull}K^*=\dim_\mathrm{Krull}H^*(BG;\Z/p).$$
\end{lemma}

\begin{proof}
Since $H^*(BT;\Z/p)$ is a finitely generated $H^*(BG;\Z/p)$-module for $p>n_\ell$, the natural map $H^*(BG;\Z/p)\to H^*(BT;\Z/p)$ is an integral extension. Since $K^*$ satisfies \eqref{condition}, it follows from Theorem \ref{AW1} and \ref{AW2} that there is an algebraic extension $K^*\to H^*(BT^k;\Z/p)$ for some $k$ which satisfies a commutative diagram of unstable algebras
$$\xymatrix{K^*\ar[r]\ar[d]^{\rm incl}&H^*(BT^k;\Z/p)\ar[d]^\phi\\
H^*(BG;\Z/p)\ar[r]&H^*(BT;\Z/p).}$$
By definition $\phi$ is obviously injective. We now consider the action of the Weyl group $W$ of $G$ on $H^*(BT;\Z/p)$. Since $H^*(BG;\Z/p)$ is a ring of invariants of $W$ and $K^*$ is its subalgebra, $K^*$ consists of invariants of $W$. Then since $K^*\to H^*(BT^k;\Z/p)$ is an algebraic extension and $H^*(BT^k;\Z/p)$ is algebraically closed by the uniqueness of Theorem \ref{AW1}, we have that $W\phi(H^*(BT^k;\Z/p))=\phi(H^*(BT^k;\Z/p))$. So the 2-dimensional part $\phi(H^2(BT^k;\Z/p))$ yields a representation of $W$. Since $k>0$ by the non-triviality of $K^*$ and the action of $W$ on $H^2(BT;\Z/p)$ is an irreducible representation, we must have $\phi(H^2(BT^k;\Z/p))=H^2(BT;\Z/p)$, implying that $\phi$ is an isomorphism. Then we obtain
$$\dim_\mathrm{Krull}K^*=\dim_\mathrm{Krull}H^*(BT^k;\Z/p)=\dim_\mathrm{Krull}H^*(BT;\Z/p)=\dim_\mathrm{Krull}H^*(BG;\Z/p),$$
where the first equality follows from Corollary \ref{dimension}.
\end{proof}

\begin{lemma}
\label{self-iso}
Suppose $G$ is simple. If a map $\varphi\colon H^*(BG;\Z/p)\to H^*(BG;\Z/p)$ of unstable algebras is non-trivial, then it is an isomorphism.
\end{lemma}

\begin{proof}
By assumption, $\mathrm{Im}\,\varphi$ is a non-trivial unstable subalgebra of $H^*(BG;\Z/p)$, so by Lemma \ref{dim}, we have $\dim_\mathrm{Krull}\mathrm{Im}\,\varphi=\dim_\mathrm{Krull}H^*(BG;\Z/p)$. The kernel of $\varphi$ is a prime ideal since $\mathrm{Im}\,\varphi$ is a domain. In particular, if $\varphi$ is not injective, $\dim_\mathrm{Krull}\mathrm{Im}\,\varphi<\dim_\mathrm{Krull}H^*(BG;\Z/p)$, a contradiction. Then $\varphi$ is injective, and hence it is an isomorphism since $H^*(BG;\Z/p)$ is of finite type.
\end{proof}

\begin{remark}
Lemma \ref{self-iso} may alternatively be proved by calculating the action of the Steenrod operations using the mod $p$ Wu formula in \cite{Sh} and the description of $H^*(BG;\Z/p)$ in \cite{HKO} for $G$ exceptional.
\end{remark}

\begin{proposition}
\label{self-equiv}
Suppose $G$ is simple and $n=n_\ell+p-1$. Any $A_n$-map $\varphi\colon G_{(p)}\to G_{(p)}$ which is non-trivial in mod $p$ cohomology is a homotopy equivalence.
\end{proposition}

\begin{proof}
By Proposition \ref{P^nX}, there is a map $\bar{\varphi}\colon P^nG_{(p)}\to BG_{(p)}$ satisfying a homotopy commutative diagram
$$\xymatrix{\Sigma G_{(p)}\ar[r]^{\Sigma\varphi}\ar[d]&\Sigma G_{(p)}\ar[d]\\
P^nG_{(p)}\ar[r]^{\bar{\varphi}}&BG_{(p)}.}$$
Then $\bar{\varphi}$ is non-trivial in the module of indecomposables of mod $p$ cohomology of dimension $\le 2n_\ell$. By Lemma \ref{lift}, there is a map of unstable algebras $\hat{\varphi}\colon H^*(BG;\Z/p)\to H^*(BG;\Z/p)$ satisfying $i_n^*\circ\hat{\varphi}=\bar{\varphi}^*$. By the non-triviality of $\bar{\varphi}^*$ above, $\hat{\varphi}$ is also non-trivial. Then by Lemma \ref{self-iso}, $\hat{\varphi}$ is an isomorphism, implying that $\varphi$ is an isomorphism in the module of indecomposables of mod $p$ cohomology. So $\varphi$ is an isomorphism in mod $p$ cohomology, and therefore since $G$ is of finite type, $\varphi$ is a homotopy equivalence by the J.H.C. Whitehead theorem.
\end{proof}

\begin{proposition}
\label{gauge-ad}
Suppose $G$ is simple, $p>2n_\ell$, $n=n_\ell+p-1$ and $\alpha$ is of infinite order. If $\G(P_\alpha)_{(p)}$ is $A_n$-equivalent to $\G(S^d\times G)_{(p)}$, then the fiberwise $p$-localized adjoint bundle $(\adjoint P_\alpha)_{(p)}$ is fiberwise $A_n$-equivalent to $S^d\times G_{(p)}$.
\end{proposition}

\begin{proof}
Let $g\colon\G(S^d\times G)_{(p)}\to\G(P_\alpha)_{(p)}$ be an $A_n$-equivalence, and define $f\colon G_{(p)}\to G_{(p)}$ by the composite
$$G_{(p)}\xrightarrow{\sigma_{(p)}}\G(S^d\times G)_{(p)}\xrightarrow{g}\G(P_\alpha)_{(p)}\xrightarrow{\pi_{(p)}}G_{(p)}$$
where $\sigma\colon G\to\G(S^d\times G)$ is the canonical section of the map $\pi\colon\G(S^d\times G)\to G$. Since $\sigma_{(p)}$ and $\pi_{(p)}$ are $A_\infty$-maps and $g$ is an $A_n$-map, $f$ is an $A_n$-map. We shall show that $f$ is a homotopy equivalence. We consider the map $\pi\colon\G(P_\beta)\to G$ in homotopy groups for a map $\beta\in\pi_d(BG)$. As in \cite{To}, we have
$$\pi_*(S^{2k-1}_{(p)})\cong\begin{cases}
\Z_{(p)}&*=2k-1\\
0&0\le*\le 2p+2k-5\text{ and }*\ne 2k-1.
\end{cases}$$
Then by \eqref{p-regular}, we get 
\begin{equation}
\label{pi(G)}
\pi_*(G_{(p)})\cong\begin{cases}
\Z_{(p)}^2&*=2m-1\text{ and }G=\Spin(4m)\\
\Z_{(p)}&*=2n_1-1,\ldots,2n_\ell-1\text{ and the above condition fails}\\
0&0\le*\le 2p-1\text{ and }*\ne 2n_1-1,\ldots,2n_\ell-1.
\end{cases}
\end{equation}
Consider the homotopy exact sequence
$$\pi_{2i-1}(\Omega^dG_{(p)})\to\pi_{2i-1}(\G(P_\beta)_{(p)})\xrightarrow{\pi_*}\pi_{2i-1}(G_{(p)})\to\pi_{2i-2}(\Omega^dG_{(p)})$$
associated with the homotopy fibration $\Omega^dG\to\G(P_\beta)\xrightarrow{\pi}G$ corresponding to the evaluation fibration $\Omega^{d-1}G\to\mathrm{map}(S^d,BG;\beta)\to BG$. Since $\alpha$ is of infinite order, $d=2n_i$ for some $i$. Then by \eqref{pi(G)} and $p>2n_\ell$, we have $\pi_{2n_\ell-1}(\Omega^dG_{(p)})=0$ and $\pi_{2n_\ell-2}(\Omega^dG_{(p)})=0$, implying that $\pi_*\colon\pi_{2n_\ell-1}(\G(P_\beta)_{(p)})\to\pi_{2n_\ell-1}(G_{(p)})$ is an isomorphism. Then in particular the maps $\sigma_*\colon\pi_{2n_\ell-1}(G_{(p)})\to\pi_{2n_\ell-1}(\G(S^d\times G)_{(p)})$ and $\pi_*\colon\pi_{2n_\ell-1}(\G(P_\alpha)_{(p)})\to\pi_{2n_\ell-1}(G_{(p)})$ are isomorphisms, and hence $f_*\colon\pi_{2n_\ell-1}(G_{(p)})\to\pi_{2n_\ell-1}(G_{(p)})$ is an isomorphism. Since the Hurewicz map $\pi_{2n_\ell-1}(G_{(p)})\to QH_{2n_\ell-1}(G_{(p)})$ is an isomorphism, it follows from the naturality of the Hurewicz map that $f_*\colon QH_{2n_\ell-1}(G_{(p)})\to QH_{2n_\ell-1}(G_{(p)})$ is an isomorphism, implying $f$ is non-trivial in mod $p$ cohomology, where $QA$ is the module of indecomposables of a ring $A$. Thus by Proposition \ref{self-equiv}, $f$ is a homotopy equivalence. Now the composite $g\circ\sigma_{(p)}\circ f^{-1}$ is an $A_n$-section of $\pi_{(p)}\colon\G(P_\alpha)_{(p)}\to G_{(p)}$. Therefore $(\adjoint P_\alpha)_{(p)}$ is fiberwise $A_n$-equivalent to $S^n\times G_{(p)}$ by Theorem \ref{A_k-section}.
\end{proof}

\begin{corollary}
\label{divisibility-3}
Suppose $G$ is simple, $p>2n_\ell$, $n=n_\ell+p-1$ and $\alpha$ is of infinite order. If $\G(P_\alpha)_{(p)}$ is $A_n$-equivalent to $\G(S^d\times G)_{(p)}$, then $\alpha$ is divisible by $p$.
\end{corollary}

\begin{proof}
Since $G$ is simple, we have $n_1\ge 2$, implying $d\ge 4$. Then the proof is done by combining Proposition \ref{divisibility-2} and \ref{gauge-ad}.
\end{proof}

We now obtain:

\begin{theorem}
\label{divisibility}
Suppose $G$ is simple, $p>2n_\ell$ and $\beta\in\pi_d(BG)$ is of infinite order. If $\alpha$ is not divisible by $p$ and $N$ is large enough, then $\G(P_\alpha)_{(p)}$ is not $A_n$-equivalent to $\G(P_{p^N\alpha})_{(p)}$ for some $n$.
\end{theorem}

\begin{proof}
Combine Proposition \ref{divisibility-1} and Corollary \ref{divisibility-3}.
\end{proof}

\begin{proof}
[Proof of Theorem \ref{main}]
Let $p_1,p_2,\ldots$ be primes greater than $2n_\ell$, and let $N_1,N_2,\ldots$ be large integers. Let $P_k$ denote the principal $G$-bundle over $S^d$ corresponding to $k$-times an infinite order generator of $\pi_d(BG)$. Then by Theorem \ref{divisibility}, $\G(P_{p_1^{N_1}\cdots p_k^{N_k}})$ for $k\ge 1$ have distinct $A_\infty$-types. Therefore the proof is completed by the infiniteness of primes.
\end{proof}

\end{document}